\documentclass[11pt]{article}
\usepackage{latexsym,bm}
\usepackage{mathrsfs}
\usepackage{amsmath,amsthm}
\usepackage{graphicx}
\usepackage{amssymb}
\usepackage{CJK}
\usepackage{color}
\usepackage{cite}
\usepackage{array}
\usepackage{stfloats}
\usepackage[colorlinks,
            linkcolor=blue,       %%修改此处为你想要的颜色
            anchorcolor=blue,  %%修改此处为你想要的颜色
            citecolor=blue,        %%修改此处为你想要的颜色，例如修改blue为red
            ]{hyperref}
\usepackage{caption}
\usepackage{graphicx}
\usepackage{epstopdf}
\usepackage{upgreek}
\usepackage[title]{appendix}

\theoremstyle{plain}
\newtheorem{thm}{Theorem}[section]
\newtheorem{cor}[thm]{Corollary}
\newtheorem{lem}[thm]{Lemma}

\theoremstyle{definition}
\newtheorem{defn}[thm]{Definition}
\theoremstyle{plain}

\theoremstyle{problem}

\theoremstyle{plain}
\newtheorem{conj}{Conjecture}

\theoremstyle{plain}
\newtheorem{cla}{Claim}
\theoremstyle{plain}

\theoremstyle{plain}

\DeclareGraphicsExtensions{.eps,.eps.gz}
\usepackage[top=2.5cm,bottom=2.5cm,left=2.8cm,right=2.2cm]{geometry}
\normalsize \rm
\allowdisplaybreaks[4]
\makeatletter
\@addtoreset{equation}{section}
\makeatother

 \usepackage{indentfirst}
\begin{document}
\begin{CJK}{GBK}{song}
\newcommand{\song}{\CJKfamily{song}}    % 宋体   (Windows自带simsun.ttf)
\newcommand{\fs}{\CJKfamily{fs}}        % 仿宋体 (Windows自带simfs.ttf)
\newcommand{\kai}{\CJKfamily{kai}}      % 楷体   (Windows自带simkai.ttf)
\newcommand{\hei}{\CJKfamily{hei}}      % 黑体   (Windows自带simhei.ttf)
\newcommand{\li}{\CJKfamily{li}}        % 隶书   (Windows自带simli.ttf)
\renewcommand\figurename{Fig.}

\begin{center}
{{\huge A solution to Godsil's conjecture on the edge-connectivity of graphs in association schemes  }} \\[18pt]% of graphs in association schemes
{\Large Wensheng Sun$^{1}$,  Yujun Yang$^{2}$,    Shou-Jun Xu$^{1,*}$  \footnotetext{*Corresponding author\\ \noindent E-mail addresses: wensheng07002@163.com(W. Sun),  yangyj@yahoo.com(Y. Yang), shjxu@lzu.edu.cn(S.-J. Xu),    }}\\[6pt]
{ \footnotesize  $^{1}$ School of Mathematics and Statistics, Gansu Center for Applied Mathematics, Lanzhou University, Lanzhou, Gansu 730000 China\\
$^{2}$ School of Mathematics and Information Science,Yantai University, Yantai 264005 China}
\end{center}

\vspace{1mm}
\begin{abstract}
A graph $G$ is called equiarboreal if the number of spanning trees containing a given edge in $G$ is independent of the choice of edge. In [Combinatorica 1(2) (1981) 163--167], Godsil proved that any graph which is a colour class in an association scheme is equiarboreal, and further conjectured that the edge-connectivity of a connected graph which is a colour class in an association scheme equals its vertex degree. In this paper, we confirm this long-standing conjecture. More generally, we prove an even stronger result that the edge-connectivity of a connected regular equiarboreal graph equals its degree by combinatorial and electrical network approaches. As a consequence, we show that every connected regular equiarboreal graph on an even number of vertices has a perfect matching.

\noindent {\bf Keywords:} Equiarboreal graph; Edge-connectivity; Association scheme; Spanning tree; Resistance distance; Principle of substitution; \\
\vspace{1mm}
\noindent{\bf AMS Classification: } 05C05, 05C12, 05C40
\end{abstract}

\section{Introduction}
A graph $G$ is called \emph{equiarboreal} if the number of spanning trees containing a given edge in $G$ is independent of the choice of edge. The concept of equiarboreal graph was first introduced by Hell and Mendelsohn, who posed the question: ``Which graphs have the property that the number of spanning trees containing a given edge is independent of the edge?" (see \cite{cdg}).  According to the definition, it follows directly that all trees and edge-transitive graphs are equiarboreal graphs$^{1}$.\footnotetext{$^{1}$ Not every vertex-transitive graph is equiarboreal, for instance, the triangular prism is vertex-transitive but not equiarboreal.} In \cite{cdg}, Godsil proved that any graph which is a colour class in an association scheme (see Definition \ref{def1}) is equiarboreal. In \cite{pfr}, Fraisse and Hell pointed out that equiarboreal graphs can be regarded as graphic matroids in which the family of all bases covers each element the same number of times.  Later, Jakobson and Rivin \cite{dja} proved that a graph is equiarboreal if and only if its weighted spanning tree is maximal in its edge deformation space.  In \cite{jzh}, Zhou, Sun and Bu gave some resistance characterizations of equiarboreal weighted and unweighted graphs, and obtained some new infinite families of equiarboreal graphs. However, it is much more difficult to characterize equiarboreal graphs for general cases.

Let $G$ be a connected graph with vertex set $V(G)$ and edge set $E(G)$. We use $\tau(G)$ to denote the number of spanning trees of $G$. For $u, v \in V(G)$, the \emph{resistance distance} \cite{djk1} between $u$ and $v$, denoted by $\Omega_G(u,v)$, is defined as the net effective resistance between the corresponding nodes in the electrical network constructed from $G$ by replacing each edge of $G$ with a unit resistor. Equivalently, we have $\Omega_G(u,v)=\frac{\tau(G_{uv})}{\tau(G)}$  \cite{bbo,cth}, where $G_{uv}$ is a graph (or multigraph) obtained from $G$ by identifying $u$ and $v$. From this viewpoint,  for an edge $e=uv \in E(G)$, the resistance distance between $u$ and $v$ represents the probability that edge $e$ is in a random spanning tree provided that all spanning trees appear with equal probability. Hence, a connected graph $G$ is equiarboreal if and only if all the resistance distances between the endpoints of each edge in $G$ are equal.  In fact, the resistance distance is closely related to random walks and spanning trees in graphs, and has been extensively studied in mathematical, physical and chemical literature \cite{lcy,liy,jzh1,ywa,mei}.

Association schemes play a unifying role in algebraic combinatorics, appearing both in error-correcting coding theory \cite{pde} and in the study of combinatorial designs, algebraic graph theory and finite group theory \cite{eba,alg}.

\begin{defn}\label{def1}
(Association scheme) Let $X$ be a finite set and $\mathcal{R}=\{R_0,R_1,...,R_n\}$ be a partition of non-empty subsets of $X \times X$. Define the matrix \( A_i \in \mathbb{R}^{|X| \times |X|} \) ($0 \leq i \leq n$) by
\[
(A_i)_{xy} = \begin{cases}
1 & \text{whenever  } (x, y) \in R_i, \\
0 & \text{otherwise}.
\end{cases}
\]
Then \( (X, \mathcal{R}) \) is called an \emph{association scheme} with \( n \) classes if the following conditions hold:\\
(i) $A_0=I$, where $I$ is the identity matrix;\\
(ii) $\sum^{n}_{i=0}A_i =J$, where $J$ is the matrix with each entry equal to one;\\
(iii) $A_{i}=A_i^{T}$ for each $i \in \{0,1,...,n\}$, where superscript $T$ denotes transposition;\\
(iv) $A_iA_j=\sum^{n}_{k=0}p^{k}_{ij}A_k$ for all $i,j \in \{0,1,...,n\}$, where $p^{k}_{ij}$ (intersection numbers) are nonnegative integers.
\end{defn}
From the above, we see that the matrices \( A_0, A_1,..., A_n \) must be linearly independent, and they generate a commutative \((n+1)\)-dimensional algebra, \(\mathcal{A}= \{A_0,A_1,...,A_n\}\), of symmetric matrices with constant diagonal. This algebra is known as the \textit{Bose-Mesner algebra} \cite{rcb}. Furthermore, since each matrix $A_i$ ($1 \leq i \leq n$) is a symmetric $(0, 1)$-matrix with zero diagonal, it can be viewed as the adjacency matrix of a graph $G_i$. In \cite{cdg}, Godsil refers to such a graph as a \emph{colour class} in an association scheme. It is well-known that any graph which is a colour class in an association scheme is regular \cite{rcb}, and any distance-regular graph is a colour class in an association scheme \cite{aeb3}. Moreover, Godsil established the following relationship.
\begin{thm}\cite{cdg}\label{thm1.2}
Any graph which is a colour class in an association scheme is equiarboreal.
\end{thm}
A nontrivial graph $G$ is \emph{$k$-edge-connected} if the removal of any $k-1$ edges does not disconnect it. The \emph{edge-connectivity} of $G$, denoted $\lambda(G)$, is the maximum value of $k$ for which $G$ is $k$-edge-connected.  Then Godsil gave a lower bound on the edge-connectivity of equiarboreal graphs and further proposed the following conjecture.
\begin{thm}\cite{cdg}\label{tm1.3}
Let $G$ be a connected equiarboreal graph on $n$ vertices with $m$ edges. Then
$$\lambda(G) \geq \dfrac{m}{n-1}.$$
\end{thm}
\begin{conj}\cite{cdg}\label{coj1}
Let $G$ be a connected graph which is a colour class in an association scheme. Then its edge-connectivity $\lambda(G)$ equals  its degree.
\end{conj}
Although Godsil's conjecture has been proposed for over forty years,  the progress on the conjecture has been slow. As early as 1972, Plesn\'{i}k \cite{jpl} showed that the edge-connectivity of a strongly regular graph is equal to its degree. In 1985, Brouwer and Mesner \cite{aeb4} proved that the vertex-connectivity of a strongly regular graph is equal to its degree and the only disconnecting sets of minimum order are the neighborhoods of its vertices. In 2005, Brouwer and Haemers \cite{aeb1} demonstrated that the edge-connectivity of a distance-regular graph equals its degree. Subsequently, Brouwer  and Koolen \cite{aeb2} proved the stronger result that the vertex-connectivity of a distance-regular graph equals its degree. These results indicate that Godsil's conjecture holds for distance-regular graphs. In 2018,  McGinnis \cite{mmc} further proved its validity  for the distance-$j$ graphs of the twisted Grassmann graphs.  However, Godsil's conjecture has remained open in the general case \cite{smc}. In this paper, we confirm this long-standing conjecture. In doing so, we study the edge-connectivity of connected regular equiarboreal graphs using combinatorial and electrical network approaches, and the main result of the paper is as follows.
\begin{thm}\label{swink}
Let $G$ be a connected $k$-regular equiarboreal graph with $k \geq 1$. Then $\lambda(G)=k$. In particular, if $k \geq 11$, the only edge cut of $k$ edges are the sets of edges incident with a single vertex.
\end{thm}
Combining  Theorems \ref{thm1.2} and \ref{swink} , we confirm Conjecture \ref{coj1} in the general case.

\begin{thm}
Let $G$ be a connected graph which is a colour class in an association scheme. Then its edge-connectivity $\lambda(G)$ equals  its degree.
\end{thm}

The remainder of this paper is structured as follows. In Section \ref{section2}, we introduce some basic concepts of graph theory, and recall useful tools and properties from electrical network theory. In Section \ref{section3}, we present the main results.  Finally, we conclude the paper in Section \ref{section4}.
\section{Preliminaries}\label{section2}
In this section,  we first introduce some basic notation that will be used later. For a graph $G$, we denote its vertex set and edge set by $V(G)$ and $E(G)$, respectively. The order and size of a graph $G$ are defined as $|V (G)|$ and $|E(G)|$, respectively.  For vertex $u\in V(G)$, we use $d_G(u)$ to stand for the degree of $u$ and $N_G(u)$ to denote the set of neighbors of $u$ in $V(G)$, so that $d_G(u) = |N_G(u)|$. The closed neighborhood of $u$ in $G$, denoted $N_G[u]$, is defined as $N_G(u) \cup \{u\}$. A subgraph $H$ of $G$ is called a spanning subgraph if $V(H) = V(G)$. For a vertex set $U \subseteq V(G)$, we use $G[U]$ to denote the subgraph of $G$ induced by $U$. For integers $n,m \geq 1$, we use $K_{n}$, $K_{m,n}$, $C_{n}$ and $S_{n}$ to denote the complete graph, the complete bipartite graph, the cycle and the star of order $n$, respectively. In particular, we use $S_{m,n}$ to denote the double star, which is the graph obtained by connecting the central vertices of two stars $S_{m+1}$ and $S_{n+1}$.

We now present the background on the tools and properties from electrical network theory that are used in this paper. An electrical network can be regarded as a weighted graph in which the weights are the resistance
of the respective edges. Therefore, it is not necessary to distinguish between electrical networks and the corresponding graphs. For convenience,  we use the notation $R_{G}(u, v)$ for the resistance (or weight) on edge $uv \in E(G)$. If each edge in a weighted graph $G$ has weight 1, we simply refer to $G$ as a graph.  We begin with the following lemma, which is known as Foster's first formula in electrical network theory.
\begin{lem}\cite{rfo}\label{lem2.1}
Let $G$ be a connected graph. Then
$$\sum_{uv\in E(G)}\Omega_G(u,v)=|V(G)|-1.$$
\end{lem}
For a connected equiarboreal graph $G$, the resistance distance is a constant between the endpoints of each edge.  For simplicity, we denote this common value by $\Omega(G)$.  By Lemma \ref{lem2.1}, the following result can be readily obtained.
\begin{lem}\label{lem2.2}
Let $G$ be a connected equiarboreal graph. Then
$$\Omega(G)=\dfrac{|V(G)|-1}{|E(G)|}.$$
\end{lem}
Below are two basic properties of resistance distances.

\textbf{Series connection:} resistors that are connected in series can be replaced by a single resistor whose resistance is the sum of the resistances.

\textbf{Parallel connection:} resistors that are connected in parallel can be replaced by a single resistor whose conductance (the inverse of resistance) is the sum of the conductances.

In \cite{ete}, two weighted graphs (networks) $G$ and $H$ are defined to be electrically equivalent with respect to $S\subseteq V(G) \cap V(H)$ if they cannot be distinguished by applying voltages to $S$ and measuring the resulting currents on $S$. From the perspective of resistance distance, we can give the following equivalent definition.

\begin{defn}
($S$-equivalent network).
Let $G$ and $H$ be two electrical networks and let $S\subseteq V(G)\cap V(H)$. If for any pair of vertices $\{u,v\} \subseteq S$, $\Omega_{G}(u, v)= \Omega_{H}(u, v)$ holds, then $G$ and $H$ are called \textit{$S$-equivalent networks}.
\end{defn}
We proceed to give a useful principle in electrical network theory.

\textbf{Principle of substitution.} Let $G$ be an electrical network with a subnetwork $H$ (not necessarily induced). If $H$ and $H^*$ are $V(H)$-equivalent networks (clearly $V(H) \subseteq V(H^*))$, then the network $G'$ obtained  by replacing $H$ with $H^*$ in $G$ is $V(G)$-equivalent to $G$. That is, for any pair of vertices $\{u,v\} \subseteq V(G)$,  we have $\Omega_{G}(u, v)= \Omega_{G'}(u, v)$.

As is well known, the series and parallel connections are commonly used equivalent substitution tools. In \cite{svg}, Gervacio provided an equivalent substitution applicable to complete bipartite graphs. To this end, he introduced the concept of negative resistances, which proves to be useful.

\textbf{Complete bipartite graph-double star transformation \cite{svg}:} The complete bipartite graph $K_{m,n}$ with partite sets $\{u_1,u_2,...,u_m\}$ and $\{v_1,v_2,...,v_n\}$ can be converted to a weighted double star $S^{\omega}_{m,n}$, as shown in Fig. \ref{Fig.6}, and the weights on the edges satisfy that
$$R_{S^{\omega}_{m,n}}(u_0,u_i)=\dfrac{1}{n}, \quad R_{S^{\omega}_{m,n}}(v_0,v_j)=\dfrac{1}{m}, \quad R^{\omega}_{S_{m,n}}(u_0,v_0)=-\dfrac{1}{nm}.$$
where $1\leq i \leq m$,  $1\leq j\leq n.$ Then $K_{m,n}$ and $S^{\omega}_{m,n}$ are $V(K_{m,n})$-equivalent networks.
\begin{figure*}[ht]
  \setlength{\abovecaptionskip}{0cm} %
  \centering
 $$\includegraphics[width=2.2in]{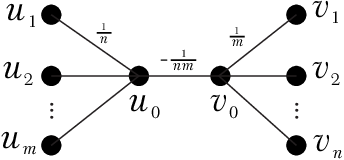}$$\\
 \caption{The equivalent double star $S^{\omega}_{m,n}$.}    \label{Fig.6}
\end{figure*}

\textbf{Rayleigh's monotonicity law}\cite{pgd}. In an electrical network, if the edge-resistance increases, then the resistance distance between any pair of vertices will not decrease.

Let $G$ be a weighted graph, and let $G^*$ be a weighted graph obtained from $G$ by identifying vertices $i$ and $j$ as a single vertex $i$ and deleting all possible loops.  Note that if $k$ is a common neighbor of $i$ and $j$ in $G$, then by parallel connection, the conductance of new edge $ki$ in $G^*$ is the sum of the conductances of edges $ik$ and $jk$ in $G$. From an electrical perspective, identifying vertices $i$ and $j$ is equivalent to short-circuiting them with a zero-resistance conductor, which corresponds in graph theory to connecting $i$ and $j$ with an edge of weight 0. Thus, by Rayleigh's monotonicity law, we have the following result.

\begin{lem}\label{lem2.4}
Let $G$ be a weighted graph and  $G^*$ be a weighted graph obtained from $G$ by identifying vertices $i$ and $j$ as a single vertex. Then for two vertices $u$, $v \in V(G^*)$, we have
$$\Omega_G(u,v) \geq \Omega_{G^*}(u,v)$$
\end{lem}

The following known degree-based lower bounds on the resistance distance between two vertices in a graph are derived by combining series and parallel connections with Lemma \ref{lem2.4}.
 \begin{lem}\label{lem2.6}\cite{dco}
Let $G$ be a graph. Then for two vertices $u$, $v \in V(G)$, we have
$$ \Omega_{G}(u,v)\geq \left\{
\begin{array}{ll}
 \dfrac{1}{d_G(u)}+ \dfrac{1}{d_G(v)}        & \mbox{if} \ uv  \notin E(G),  \\

 \dfrac{1}{d_G(u)+1}+ \dfrac{1}{d_G(v)+1}         & \mbox{if} \ uv \in E(G).
\end{array}
\right.$$
\end{lem}

For a weighted graph $G$ and $u \in V(G)$,  we denote by $W_G(u)$ the sum of the inverse of the weights on the edges incident to $u$, that is
$$W_G(u)=\sum_{v \in N_G(u)}\dfrac{1}{R_G(u,v)}.$$
In particular, $W_G(u)=d_G(u)$ if $G$ is unweighted.
\begin{lem}\label{lem2.5}
Let $G$ be a weighted graph. Then for two vertices $u$, $v \in V(G)$, we have
$$\Omega_{G}(u,v)\geq \emph{max}\left\{\dfrac{1}{W_G(u)}, \dfrac{1}{W_G(v)}\right\}.$$
\end{lem}
\begin{proof}Let $G^*$ be the weighted graph obtained from $G$ by identifying all vertices in $V(G)\setminus \{u\}$ as a single vertex $w$, which results in a weighted complete graph $K^{\omega}_2$ with edge weight $\frac{1}{W_G(u)}$.  Then, from Lemma \ref{lem2.4}, it follows that
$$\Omega_G(u,v) \geq \Omega_{G^*}(u,v) = \dfrac{1}{W_G(u)}.$$
Similarly, we can obtain $\Omega_G(u,v) \geq \frac{1}{W_G(v)}$.
\end{proof}
%By parallel connection, we have $\Omega_{G^*}(u,w) = \frac{1}{W_G(u)}$.

\section{Main results}\label{section3}

Let $G$ be a nontrivial connected graph. An \emph{edge cut} in $G$ is a set of edges $C \subseteq E(G)$ of the form $C = \{ uv \in E(G) \mid u \in A,\ v \in B \}$, where $(A, B)$ is a partition of $V(G)$ with $A \neq \emptyset$ and $B \neq \emptyset$. The edge-connectivity of $G$ is the minimum cardinality over all such edge cuts. The edge cut $C$ is called \emph{non-trivial} if $|A| \geq 2$ and $|B| \geq 2$. Let $G_1 = G[A]$ and $G_2 = G[B]$ denote the subgraphs induced by $A$ and $B$, respectively.  Let $k = |C|$ be the size of the edge cut. We label the edges in $C$ as $e_i=u_iv_i$ ($1 \leq i \leq k$), where $u_i$ is a vertex in $A$ and $v_i$ is a vertex in $B$. Note that the vertices $u_i$ (or $v_i$) are not necessarily distinct, as different edges may share the same endpoint in $A$ (or $B$).  In the following discussion, we also denote by $G[C]$ the edge-induced subgraph of $G$ with edge set $C$. Clearly, in this viewpoint, $G[C]$ is a bipartite graph, and its bipartition is given by $A \cap V(G[C])$ and $B \cap V(G[C])$, which we denote by $A_1$ and $B_1$, respectively, so that $A_1 \subseteq A$ and $B_1 \subseteq B$.  We first give the following result.
\begin{lem}\label{lem3.1}
Let $G$ be a connected $k$-regular equiarboreal graph with $k \geq 3$.  Let $C$ be a non-trivial edge cut satisfying $|C| \leq k$. Then
$|A_1| \geq 2$ and $|B_1| \geq 2$.
\end{lem}
\begin{proof} Here we only prove that $|A_1| \geq 2$. The  case  $|B_1| \geq 2$ can be proved in the same way. Suppose, to the contrary, that $A_1 = \{u_1\}$.  According to Theorem \ref{thm1.2}, we have
$$\lambda(G) \geq \frac{|E(G)|}{|V(G)|-1} = \frac{k|V(G)|}{2(|V(G)|-1)}> \frac{k}{2}.$$
Since $C$ is a non-trivial edge cut, we have $\frac{k}{2} < |C| \leq k-1$. Define the edge set $C'=\{u_1u\in E(G)\mid  u\in A\}$, then $|C'|=k-|C|$. Therefore, we have $1 \leq |C'| < \frac{k}{2}$. On the other hand, it is not hard to see that $C'$ is an edge cut of $G$, which contradicts $\lambda(G) > \frac{k}{2}.$
\end{proof}

\begin{lem}\cite{bbo}\label{lem3.3}
Let $G$ be a connected $k$-regular graph. If $k$ is even, then $\lambda(G)$ is even.
\end{lem}

For a connected $k$-regular equiarboreal graph $G$, define a function $\mathcal{F}(x,y)$ for $x,y \geq 1$, as follows:
$$\mathcal{F}(x,y)=\dfrac{4(k-x)(k-y)-(k-x-y+1)^2}{2(k-x)(k-y)(k+1)-k(k-x-y+1)^{2}}.$$
Moreover, we give a lower bound for $\Omega(G)$ via this function by proving the following theorem.  For convenience, the weight of all unlabeled edges in the figures is set to 1.
\begin{thm}\label{tm3.66}
Let $G$ be a connected $k$-regular equiarboreal graph with $k \geq 3$. Let $C$ be a non-trivial edge cut satisfying $|C| \leq k$. Then
$$\Omega(G) \geq \emph{max} \Big\{\mathcal{F}(d_{G[C]}(u),d_{G[C]}(v))| uv \in C\Big\}.  $$
\end{thm}
\begin{proof}
Let $C = \{uv \in E(G)\mid u\in A, v\in B \}$ be a non-trivial edge cut of $G$ with $|C| \leq k$.  Without loss of generality, arbitrarily  choose an edge $e=(u_1,v_1) \in C$. Suppose that $d_{G[C]}(u_1)=x$ and $d_{G[C]}(v_1)=y$, where $1 \leq x,y \leq k-1.$  Our main goal is to prove the following inequality:
$$\Omega(G) \geq \mathcal{F}(x,y).$$
By Lemma \ref{lem3.1}, we have $|A_1| \geq 2$ and $|B_1| \geq 2$. As an example, the structure of $G[C]$ is shown in Fig. \ref{Fig.5}(a). Let $G^*$ be the weighted graph obtained from $G$ by identifying the vertex sets $A_1 \setminus \{u_1\}$ and $B_1 \setminus \{v_1\}$  as vertices $u_2$ and $v_2$, respectively. According to Lemma \ref{lem2.4}, we have
\begin{equation}\label{eq3.11}
\Omega(G) =\Omega_G(u_1,v_1) \geq \Omega_{G^*}(u_1,v_1).
\end{equation}
\begin{figure*}[ht]
  \setlength{\abovecaptionskip}{0cm} %
  \centering
 $$\includegraphics[width=4.6in]{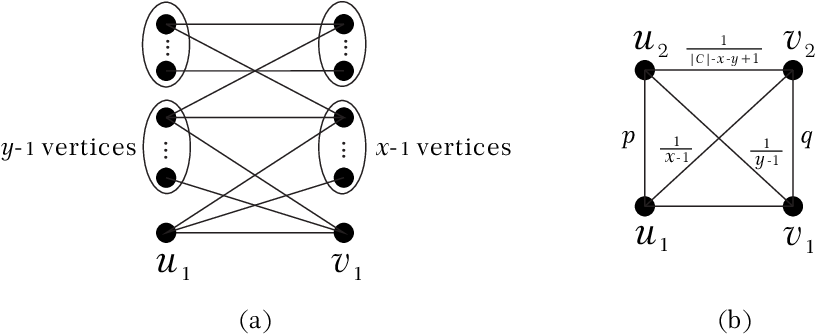}$$\\
 \caption{(a) The structure example of $G[C]$, (b) The weighted graph $K^{\omega}_4$ equivalent to $G^*$.}    \label{Fig.5}
\end{figure*}\\
Observe that $C^{*}=\{u_iv_j| 1 \leq i,j \leq 2\}$ is an edge cut of $G^*$. Assume that
$$G^*_1=G^*[(A\setminus A_1) \cup \{u_1,u_2\}], \quad G^*_2=G^*[(B\setminus B_1) \cup \{v_1,v_2\}].$$
Furthermore, by the principle of substitution, the graph $G^*$ can be transformed to a weighted complete graph $K^{\omega}_4$ as shown in Fig. \ref{Fig.5}(b), where $G$ and $K^{\omega}_4$ are $\{u_1,u_2,v_1,v_2\}$-equivalent networks, with edge weights satisfying:
$$R_{K^{\omega}_4}(u_1,v_1)=1, \quad R_{K^{\omega}_4}(u_1,u_2)= \Omega_{G^*_1}(u_1,u_2)=p, \quad   R_{K^{\omega}_4}(v_1,v_2)= \Omega_{G^*_2}(v_1,v_2)=q,$$
$$ \ R_{K^{\omega}_4}(u_1,v_2)=\frac{1}{x-1}, \quad R_{K^{\omega}_4}(u_2,v_1)= \frac{1}{y-1}, \quad R_{K^{\omega}_4}(u_2,v_2)= \frac{1}{|C|-x-y+1}, $$
where $1 \leq x,y \leq k-1$, and $2 \leq x+y \leq |C|+1 $.\\
Here the cases for $x=1$, $y=1$ or $x+y=|C|+1$ are allowed so that the edges $u_1v_2$, $u_2v_1$ or $u_2v_2$ always exist in $G[C]$ since the cases that $u_1$ and $v_2$, $u_2$ and $v_1$, or $u_2$ and $v_2$ are not adjacent in $G[C]$ can be viewed as they are connected by an edge of weight $+\infty$, respectively. Observe that $W_{G^*_1}(u_1)=k-x$ and $W_{G^*_2}(v_1)=k-y$, according to Lemma \ref{lem2.5}, we get
$$ \Omega_{G^*_1}(u_1,u_2) \geq \frac{1}{W_{G^*_1}(u_1)}=\frac{1}{k-x}, \quad   \Omega_{G^*_2}(v_1,v_2) \geq \dfrac{1}{W_{G^*_1}(v_1)}=\dfrac{1}{k-y}.     $$
Now we construct a corresponding weighted graph $N$ from $K^{\omega}_4$ by assigning the following weights to its edges:
$$R_{N}(u_1,v_1)=1, \quad  R_{N}(u_1,u_2)= \dfrac{1}{k-x}, \quad  R_{N}(v_1,v_2)= \dfrac{1}{k-y},$$
$$ \ R_{N}(u_1,v_2)=\frac{1}{x-1}, \quad R_{N}(u_2,v_1)= \frac{1}{y-1}, \quad R_{N}(u_2,v_2)= \frac{1}{k-x-y+1}, $$
%$$p=\dfrac{1}{k-x}, \quad q=, \quad R_{N}(u_2,v_2)= \dfrac{1}{k-x-y+1}.$$
By applying the principle of substitution and Rayleigh's monotonicity law, we obtain
\begin{equation}\label{eq3.20}
\Omega_{G^*}(u_1,v_1) = \Omega_{K^{\omega}_4}(u_1,v_1) \geq \Omega_{N}(u_1,v_1).
\end{equation}
Then we focus on calculating the resistance distance between $u_1$ and $v_1$ in $N$ by using Kirchhoff's current law. We impose a unit voltage between vertices $u_1$ and $v_1$,  and let $V_i$ be the voltage at vertex $i \in\{u_1,u_2,v_1,v_2\}$. Thus, $V_{u_1}=1$ and $V_{v_1}=0$.  From Ohm's law,  let the current entering vertex $u_1$ be $I$, and the current leaving vertex $v_1$ be $I$. Since Kirchhoff's current law states that the total current outflow any vertex is 0, we obtain the following equations for each vertex.  For vertex $u_1$, we have
$$\dfrac{1-V_{u_2}}{\dfrac{1}{k-x}}+\dfrac{1-V_{v_2}}{\dfrac{1}{x-1}}+1=I,$$
which yields that
\begin{equation}\label{eq3.12}
(k-x)V_{u_2}+(x-1)V_{v_2}=k-I.
\end{equation}
Similarly, for vertex $v_1$, we have
\begin{equation*}
(y-1)V_{u_2}+(k-y)V_{v_2}=I-1.
\end{equation*}
For vertex $u_2$, we have
$$\dfrac{V_{u_2}}{\dfrac{1}{y-1}}+\dfrac{V_{u_2}-V_{v_2}}{\dfrac{1}{k-x-y+1}} = \dfrac{1-V_{u_2}}{\dfrac{1}{k-x}}.$$
By simplifying, we obtain
\begin{equation}\label{eq3.13}
2(k-x)V_{u_2}-(k-x-y+1)V_{v_2}=k-x.
\end{equation}
Similarly, for vertex $v_2$, we have
\begin{equation}\label{eq3.14}
2(k-y)V_{v_2}-(k-x-y+1)V_{u_2}=x-1.
\end{equation}
For convenience, let $a=k-x$, $b=k-y$, and $c=k-x-y+1$. Then, we can establish a system of equations based on Eqs. (\ref{eq3.13}) and (\ref{eq3.14}), as follows.
$$ \begin{cases}
2aV_{u_2}-cV_{v_2}=a,\\
2bV_{v_2}-cV_{u_2}=x-1.
\end{cases}
$$
Solving the above system of equations, we obtain
\begin{equation}\label{eq3.15}
 V_{u_2}=\dfrac{2ab+c(x-1)}{4ab-c^2}, \  V_{v_2}=\dfrac{a(2x+c-2)}{4ab-c^2}.
 \end{equation}
Substituting Eq. (\ref{eq3.15}) into Eq. (\ref{eq3.12}), we have
\begin{equation}\label{eq3.88}
I=k-(aV_{u_2}+(x-1)V_{v_2}).
\end{equation}
By the simplification of Eq. (\ref{eq3.88}) in Appendix \ref{appb}, we obtain
\begin{equation}\label{eq3.99}
I=\dfrac{2ab(k+1)-kc^2}{4ab-c^2}.
\end{equation}
As consequence, we have
$$\Omega_N(u_1,v_1)=\dfrac{1}{I}=\dfrac{4ab-c^2}{2ab(k+1)-kc^2},  $$
which equivalent to
$$\Omega_N(u_1,v_1)=\dfrac{4(k-x)(k-y)-(k-x-y+1)^2 }{2(k-x)(k-y)(k+1)-k(k-x-y+1)^{2} }=\mathcal{F}(x,y).  $$
Finally, the desired result follows from Ineqs. (\ref{eq3.11}) and (\ref{eq3.20}). The proof is complete.
\end{proof}

Next, we study the edge-connectivity of connected $k$-regular equiarboreal graphs for $k \geq 1$. The cases $k = 1$ and $k = 2$ are trivial. We therefore focus on $k \geq 3$ in sequences.
\subsection{Cubic equiarboreal  graphs and $4$-regular equiarboreal graphs}
\begin{thm}\label{tm3.1}
Let $G$ be a connected cubic equiarboreal  graph. Then $\lambda(G)=3$.
\end{thm}
\begin{proof}
Let $G$ be a connected cubic equiarboreal graph. By Theorem \ref{tm1.3}, we have $\lambda(G) \geq 2.$
Hence we only need to prove that $\lambda(G) \neq 2$. By contradiction, suppose that there exists an edge cut $C=\{(u_1,v_1),(u_2,v_2)\}$ of $G$ with $|C|=2$. We consider the following two cases.

\textbf{Case 1.} $u_1=u_2$ or $v_1=v_2$.  This case does not hold directly by Lemma \ref{lem3.1}.

\textbf{Case 2.} $u_1 \neq u_2$ and $v_1 \neq v_2$. It can be seen that graph $G_1$ is $\{u_1,u_2\}$-equivalent to an edge $u_1u_2$ with weight $\Omega_{G_1}(u_1,u_2)$, and graph $G_2$ is $\{v_1,v_2\}$-equivalent to an edge $v_1v_2$ with weight $\Omega_{G_2}(v_1,v_2)$. By the principle of substitution,  the graph $G$ can be transformed into a weighted 4-cycle $C^{\omega}_4:u_1u_2v_2v_1$, with the weights satisfying:
$$R_{C^{\omega}_4}(u_1,v_1)=R_{C^{\omega}_4}(u_2,v_2)=1, \quad R_{C^{\omega}_4}(u_1,u_2)= \Omega_{G_1}(u_1,u_2), \quad R_{C^{\omega}_4}(v_1,v_2)= \Omega_{G_2}(v_1,v_2).$$
Then $G$ and $C^{\omega}_4$ are $\{u_1, u_2, v_1, v_2\}$-equivalent networks.\\
Note that the degrees of $u_1$ and $u_2$ are equal to 2 in $G_1$. By Lemma \ref{lem2.6}, we know that
$$ \Omega_{G_1}(u_1,u_2) \geq \frac{1}{d_{G_1}(u_1)+1}+\frac{1}{d_{G_1}(u_2)+1}=\frac{2}{3}.$$
Similarly, we have $ \Omega_{G_2}(v_1,v_2)\geq \frac{2}{3}.$ Due to the fact that $G$ is an equiarboreal graph, by Rayleigh's monotonicity law, we have
\begin{align}\label{eq3.1}
\Omega(G)&=\Omega_{G}(u_1,v_1)= \Omega_{C^{\omega}_4}(u_1,v_1)   \notag\\
&= \frac{R_{C^{\omega}_4}(u_1,u_2)+R_{C^{\omega}_4}(v_1,v_2)+1}{R_{C^{\omega}_4}(u_1,u_2)+R_{C^{\omega}_4}(v_1,v_2)+2}         \notag\\
&\geq   \frac{\frac{2}{3}+\frac{2}{3}+1}{\frac{2}{3}+\frac{2}{3}+2}           \notag\\
&=\frac{7}{10}.
\end{align}
On the other hand,  by Lemma \ref{lem2.2}, we have
$$\Omega(G)=\frac{|V(G)|-1}{|E(G)|}=\frac{2(|V(G)|-1)}{3|V(G)|} < \frac{2}{3}, $$
which yields a contradiction to Ineq. (\ref{eq3.1}).
Therefore, we conclude that $\lambda(G) \neq 2$ and complete the proof.
\end{proof}

For a cubic graph, the edge-connectivity is equal to its vertex-connectivity. Hence, Theorem \ref{tm3.1} leads to the following result.
\begin{cor}
Let $G$ be a connected cubic equiarboreal  graph. Then its vertex-connectivity equals to 3.
\end{cor}
\begin{thm}\label{tm3.55}
Let $G$ be a connected $4$-regular equiarboreal graph. Then $\lambda(G)=4$.
\end{thm}
\begin{proof}
By Theorem \ref{tm1.3}, we have $\lambda(G) > 2.$ Then $\lambda(G)=3$ is invalid by Lemma \ref{lem3.3}. Hence, we have $\lambda(G)=4$.
\end{proof}

\subsection{$5$-regular equiarboreal graphs}

\begin{figure*}[ht]
  \setlength{\abovecaptionskip}{0cm} %
  \centering
 $$\includegraphics[width=5.2in]{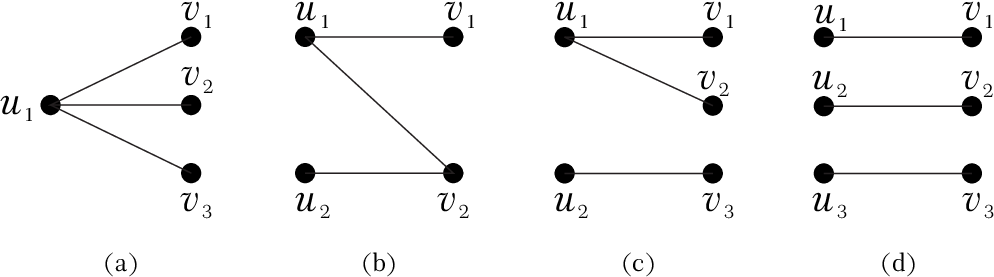}$$\\
 \caption{The corresponding graphs $G[C]$ for edge cuts $C$ with $|C|=3$.}    \label{Fig.1}
\end{figure*}

For a connected graph $G$, when the edge-connectivity $\lambda(G)=3$,  we can verify that there exist four types of edge cuts $C$ satisfying $|C|=3$, whose corresponding graphs $G[C]$ are shown in Fig. \ref{Fig.1}. When $\lambda(G)=4$, there exist ten types of edge cuts $C$ satisfying $|C|=4$, whose corresponding graphs $G[C]$ are shown in Fig. \ref{Fig.2}. In particular, $G[C]$ is called \emph{$K_2$ component-free} if it has no connected component isomorphic to $K_2$. We present the following result.

\begin{thm}\label{tm3.5}
Let $G$ be a connected $k$-regular equiarboreal graph with $k \geq 4$. Let $C$ be a non-trivial edge cut satisfying $|C| \leq k$. Then $G[C]$ is $K_2$ component-free.
\end{thm}
\begin{proof}
Suppose that $C = \{uv \in E(G) \mid u \in A, v \in B \}$ is a non-trivial edge cut with $|C| \leq k$. We argue by contradiction. Suppose that $G[C]$ is not $K_2$ component-free and let the edge $u_1v_1$ form a $K_2$ component. Then,  $d_{G[C]}(u_1)=d_{G[C]}(v_1)=1$. By Theorem \ref{tm3.66}, we have
\begin{equation*}\label{eq3.3}
\Omega(G) \geq \mathcal{F}(d_{G[C]}(u_1),d_{G[C]}(v_1)) = \mathcal{F}(1,1)= \frac{3}{k+2}.
\end{equation*}
Since $G$ is an equiarboreal graph, it follows by Lemma \ref{lem2.2} that
\begin{equation*}
\Omega(G)=\Omega_G(u_1,v_1) =\frac{2(|V(G)|-1)}{k|V(G)|}<\frac{2}{k},
\end{equation*}
achieving a desired contradiction due to $k \geq 4$, and complete the proof.
\end{proof}

\begin{figure*}[ht]
  \setlength{\abovecaptionskip}{0cm} %
  \centering
 $$\includegraphics[width=6.1in]{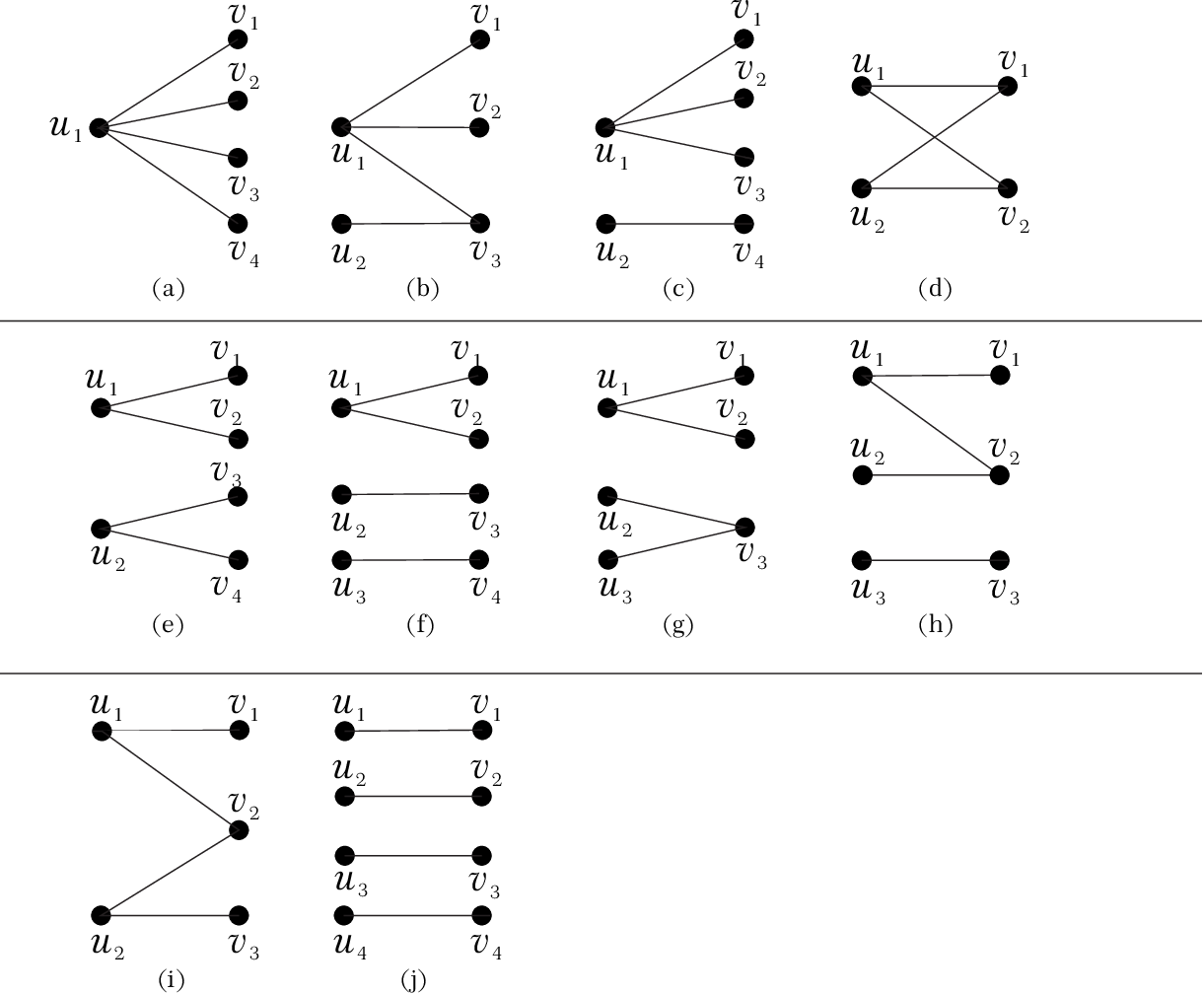}$$\\
 \caption{The corresponding graphs $G[C]$ for edge cuts $C$ with $|C|=4$.}    \label{Fig.2}
\end{figure*}

\begin{thm}\label{tm3.6}
Let $G$ be a connected $5$-regular equiarboreal graph. Then $\lambda(G)=5$.
\end{thm}
\begin{proof}
Let $G$ be a connected 5-regular equiarboreal graph. By Theorem \ref{tm1.3}, we have $\lambda(G) \geq 3.$ In the following, we show that $\lambda(G) \neq 3$ and $\lambda(G) \neq 4$ in sequence. First, for a contradiction, suppose $\lambda(G)= 3$. Then there exist four types of edge cuts $C$ satisfying $|C|=3$, as shown in Fig. \ref{Fig.1}. According to Theorem \ref{tm3.5}, $G[C]$ does not contain a $K_2$ component. Therefore, only the two cases depicted in Fig. \ref{Fig.1}(a) and (b) need to be considered.

\textbf{Case 1.} $G$ contains an edge cut as shown in Fig. \ref{Fig.1}(a). By Lemma \ref{lem3.1}, this case can not hold.

\begin{figure*}[ht]
  \setlength{\abovecaptionskip}{0cm} %
  \centering
 $$\includegraphics[width=5in]{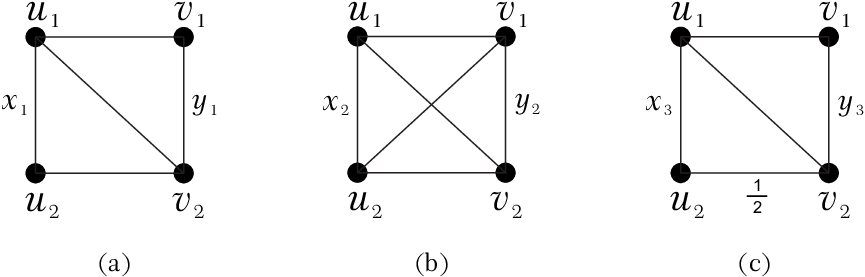}$$\\
 \caption{(a) The network $M_1$, (b) The network $M_2$, (c) The network $M_3$.}    \label{Fig.3}
\end{figure*}

\textbf{Case 2.} $G$ contains an edge cut as shown in Fig. \ref{Fig.1}(b). By the principle of substitution, $G$ can be transformed into an equivalent network $M_1$ as shown in Fig. \ref{Fig.3}(a), and the weights satisfy
$$R_{M_1}(u_1,v_1)=R_{M_1}(u_1,v_2)=R_{M_1}(u_2,v_2)=1,$$
$$R_{M_1}(u_1,u_2)=\Omega_{G_1}(u_1,u_2)=x_1, \quad R_{M_1}(v_1,v_2)=\Omega_{G_2}(v_1,v_2)=y_1.$$
 By the series and parallel connections, we have
\begin{align}
\Omega_{G}(u_1,v_1)- \Omega_{G}(u_1,v_2)&= \Omega_{M_1}(u_1,v_1)- \Omega_{M_1}(u_1,v_2)  \notag\\
&= \frac{(1+x_1)+(2+x_1)y_1}{x_1y_1+2x_1+2y_1+3} -\frac{(1+x_1)(1+y_1)}{x_1y_1+2x_1+2y_1+3}        \notag\\
&= \frac{y_1}{x_1y_1+2x_1+2y_1+3}          \notag\\
&>0, \ (\mbox{since} \ x_1,y_1 >0)\notag
\end{align}
which contradicts the property of $G$ as an equiarboreal graph.

Hence, we conclude that $\lambda(G) \neq 3$. Again by contradiction,  assume that $\lambda(G)= 4$. Then there exist ten types of edge cuts $C$ satisfying $|C|=4$, as shown in Fig \ref{Fig.2}.
By Theorem \ref{tm3.5}, it suffices to consider the six edge cut cases as shown in Fig. \ref{Fig.2}(a), (b), (d), (e), (g) and (i). On the other hand, according to Lemma \ref{lem2.2}, we have
\begin{equation}\label{eq3.4}
\Omega(G)=\frac{|V(G)|-1}{|E(G)|}=\frac{2(|V(G)|-1)}{5|V(G)|} < \frac{2}{5}.
\end{equation}
\textbf{Case 1.} $G$ contains an edge cut as shown in Fig. \ref{Fig.2}(a). This case does not hold directly from Lemma \ref{lem3.1}.

\textbf{Case 2.} $G$ contains an edge cut as shown in Fig. \ref{Fig.2}(b). Let $G_3$ be a subgraph induced by vertex set $B \cup  \{u_1\}$ of $G$. Since graph $G_1$ is $\{u_1,u_2\}$-equivalent to an edge $u_1u_2$ with weight $\Omega_{G_1}(u_1,u_2)$, and graph $G_3$ is $\{u_1,v_3\}$-equivalent to an edge $u_1v_3$ with weight $\Omega_{G_3}(u_1,v_3)$, it follows from the principle of substitution that $G$ can be transformed into a weighted 3-cycle $C^{\omega}_3: u_1u_2v_3$, and its edge weights satisfy
$$R_{C^{\omega}_3}(u_1,v_2)=\Omega_{G_1}(u_1,u_2), \quad R_{C^{\omega}_3}(u_2,v_3)= 1, \quad R_{C^{\omega}_3}(u_1,v_3)= \Omega_{G_3}(u_1,v_3).$$
Then $G$ and $C^{\omega}_3$ are $\{u_1,u_2,v_3\}$-equivalent networks.\\
Since $G$ is an equiarboreal graph, we have
$$ \Omega(G)=\Omega_G(u_2,v_3)=\Omega_{C^{\omega}_3}(u_2,v_3)=  \frac{\Omega_{G_1}(u_1,u_2)+\Omega_{G_3}(u_1,v_3)}{\Omega_{G_1}(u_1,u_2)+\Omega_{G_3}(u_1,v_3)+1}.$$
Observe that in $G_1$ the degrees of $u_1$ and $u_2$ are $2$ and $4$, respectively. By Lemma \ref{lem2.6}, we know that
$$ \Omega_{G_1}(u_1,u_2) \geq \frac{1}{d_{G_1}(u_1)+1}+\frac{1}{d_{G_1}(u_2)+1}=\frac{8}{15}.$$
Similarly,  $ \Omega_{G_3}(u_1,v_3)\geq \frac{9}{20}$ holds. By Rayleigh's monotonicity law, we have
$$ \Omega(G) \geq \frac{\frac{8}{15}+\frac{9}{20}}{\frac{8}{15}+\frac{9}{20}+1}= \frac{59}{119},$$
a contradiction to Ineq. (\ref{eq3.4}).

\textbf{Case 3.} $G$ contains an edge cut as shown in Fig. \ref{Fig.2}(d). By the principle of substitution, $G$ can be transformed into an equivalent weighted complete graph $M_2$ as shown in Fig. \ref{Fig.3}(b), and the weights satisfy
$$R_{M_2}(u_1,v_1)=R_{M_2}(u_1,v_2)=R_{M_2}(u_2,v_1)=R_{M_2}(u_2,v_2)=1,$$
$$R_{M_2}(u_1,u_2)=\Omega_{G_1}(u_1,u_2)=x_2, \quad R_{M_2}(v_1,v_2)=\Omega_{G_2}(v_1,v_2)=y_2.$$
Note that $M_2$ contains a subgraph $K_{2,2}$ with partite sets $\{u_1,u_2\}$ and $\{v_1,v_2\}$. By using the complete bipartite graph-double star transformation, it is not difficult to obtain
$$\Omega(G)=\Omega_G(u_1,v_1)=\Omega_{M_2}(u_1,v_1)=\frac{1+2x_2}{4+4x_2}+\frac{1+2y_2}{4+4y_2}-\frac{1}{4}.$$
Then, the degrees of $u_1$ and $u_2$ are equal to $3$ in $G_1$. According to Lemma \ref{lem2.6}, we get
$$ \Omega_{G_1}(u_1,u_2) \geq \frac{1}{d_{G_1}(u_1)+1}+\frac{1}{d_{G_1}(u_2)+1}=\frac{1}{2}.$$
Similarly, we can get $ \Omega_{G_2}(u_2,v_2)\geq \frac{1}{2}$. By Rayleigh's monotonicity law, we have
$$ \Omega(G) \geq \frac{1}{3}+\frac{1}{3}-\frac{1}{4} =\frac{5}{12},$$
a contradiction to Ineq. (\ref{eq3.4}).

\textbf{Case 4.} $G$ contains an edge cut as shown in Fig. \ref{Fig.2}(e). Let $G^*$ be the weighted graph obtained from $G$ by identifying vertices $\{v_2,v_3,v_4\}$  as vertex $v_2$. By Lemma \ref{lem2.4}, we have
\begin{equation}\label{eq3.5}
\Omega(G)=\Omega_{G}(u_1,v_1)\geq\Omega_{G^*}(u_1,v_1).
\end{equation}
It can be seen that $C^{*}=\{u_1v_1, u_1v_2,u_2v_2\}$ is an edge cut of graph $G^*$. Suppose that $G^*_1=G^*[A]$ and $G^*_2=G^*[B\setminus \{v_3,v_4\}]$. By the principle of substitution, the graph $G^*$ can be reduced to an equivalent network $M_3$ as shown in Fig. \ref{Fig.3}(c), and the weights satisfy
$$R_{M_3}(u_1,v_1)=R_{M_3}(u_1,v_2)=1, \quad R_{M_3}(u_2,v_2)=\frac{1}{2}.$$
$$R_{M_3}(u_1,u_2)= \Omega_{G^*_1}(u_1,u_2)=x_3, \quad R_{M_3}(v_1,v_2)= \Omega_{G^*_2}(v_1,v_2)=y_3.$$
According to Lemma \ref{lem2.5}, for graphs $G^*_1$ and $G^*_2$, we have
$$ \Omega_{G^*_1}(u_1,u_2) \geq \frac{1}{W_{G^*_1}(u_1)}=\frac{1}{3}, \quad    \Omega_{G^*_2}(v_1,v_2) \geq \frac{1}{W_{G^*_2}(v_1)}=\frac{1}{4}.     $$
Then by Rayleigh's monotonicity law, we show that
\begin{align}\label{eq3.6}
\Omega_{G^*}(u_1,v_1)&= \Omega_{M_3}(u_1,v_1)   \notag\\
&= \frac{2x_3+y_3(2x_3+3)+1}{2x_3+(2x_3+3)(y_3+1)+1}         \notag\\
&\geq  \dfrac{\frac{2}{3}+\frac{1}{4}(3+\frac{2}{3})+1}{\frac{2}{3}+\frac{5}{4}(3+\frac{2}{3})+1}\notag\\
&=\frac{31}{75}.
\end{align}
Comparing Ineqs. (\ref{eq3.5}) and (\ref{eq3.6}), we have
\begin{equation*}
\Omega(G) \geq \frac{31}{75},
\end{equation*}
a contradiction to Ineq. (\ref{eq3.4}).

\textbf{Case 5.} $G$ contains an edge cut as shown in Fig. \ref{Fig.2}(g).  For this case, we can identify vertices $\{u_2,u_3\}$ and $\{v_2,v_3\}$ of $G$  as vertices $u_2$ and $v_2$, respectively, then obtain the graph $G^*$.  Similar to the analysis in Case 4, $G^*$ can be simplified to the same network $M_3$ shown in Fig. \ref{Fig.3}(c), which leads to a contradiction of Ineq.~(\ref{eq3.4}).

\textbf{Case 6.} $G$ contains an edge cut as shown in Fig. \ref{Fig.2}(i).  For this case, we can identify vertices $\{v_2,v_3\}$  of $G$  as vertex $v_2$, and obtain the graph $G^*$.  Similarly, $G^*$ can be simplified to the same network $M_3$ shown in Fig. \ref{Fig.3}(c), leading to a contradiction.

Based on the above analysis, $\lambda(G) \neq 3$ and $\lambda(G) \neq 4$. Thus, we conclude that $\lambda(G) =5$, which completes the proof.
\end{proof}

In the following, we focus on the $k$-regular equiarboreal  graphs with $k \geq 6$, for which some results are presented. Let $C$ be an edge cut of $G$. For integers $x,y \geq 1$, $G[C]$ is called \emph{strongly $S_{x,y}$-free}, if there does not exist edge $(u,v) \in C$ such that $d_{G[C]}(u) = x+1$ and $d_{G[C]}(v) = y+1$. For example, the graph as shown in Fig. \ref{Fig.2}(b) is not strongly $S_{2,1}$-free.

\begin{thm}\label{tm3.13}
Let $G$ be a connected $k$-regular equiarboreal graph with $k \geq 7$. Let $C$ be a non-trivial edge cut satisfying $|C| \leq k$. Then for all integers $x,y \geq 1$ satisfying $ x+y \leq k- \sqrt{k}-2$, $G[C]$ is strongly $S_{x,y}$-free.
\end{thm}
\begin{proof}
Let $C = \{uv\in E(G)\mid u\in A, v\in B \}$ be a non-trivial edge cut of $G$ with $|C| \leq k$. Suppose to the contrary that there exist two integers $x_1,y_1 \geq 1$ satisfying $ x_1+y_1 \leq k- \sqrt{k}-2$ such that $G[C]$ is not strongly $S_{x_1,y_1}$-free. Thus, there exists an edge $e=(u_1,v_1) \in C$ such that $d_{G[C]}(u_1)=x_1+1$ and $d_{G[C]}(v_1)=y_1+1$. By Theorem \ref{tm3.66}, we have
\begin{equation}\label{eq3.66}
\Omega(G) \geq \mathcal{F}(d_{G[C]}(u_1),d_{G[C]}(v_1)) = \mathcal{F}(x_1+1,y_1+1).
\end{equation}
We state the following claim.
\begin{cla}\label{claim4}
For integers $x,y \geq 1$ satisfying $ x+y \leq k- \sqrt{k}-2$, then $\mathcal{F}(x+1,y+1) \geq \frac{2}{k}$.
\end{cla}
\emph{Proof of Claim \ref{claim4}.} By computing, we have
\begin{equation}\label{eq3.16}
\mathcal{F}(x+1,y+1)-\frac{2}{k}=\dfrac{k(k-x-y-1)^2-4(k-x-1)(k-y-1)}{k[2(k-x-1)(k-y-1)(k+1)-k(k-x-y-1)^{2}]},
\end{equation}
where the denominator of Eq. (\ref{eq3.16}) is strictly  greater than 0 (see its justification in Appendix \ref{appa}). To prove Claim \ref{claim4}, it is sufficient to show that
\begin{equation}\label{eq3.17}
k(k-x-y-1)^2-4(k-x-1)(k-y-1)\geq 0,
\end{equation}
Let $a=k-x-1$, $b=k-y-1$. Then Ineq. (\ref{eq3.17}) can be written as
\begin{equation*}
k(a+b-k+1)^2\geq 4ab.
\end{equation*}
Let $t=a+b$.  We know that  the maximum value of $ab$ is $\frac{t^2}{4}$ (achieved when $a=b=\frac{t}{2}$). Thus, if we can show
\begin{equation}\label{eq3.177}
k(t-k+1)^2\geq t^2,
\end{equation}
then Ineq. (\ref{eq3.17}) naturally holds.
Since $x+y \leq k- \sqrt{k}-2$, we have $t \geq k+\sqrt{k}$. Set $s=t-k+1$, consequently $s \geq \sqrt{k}+1$. Then Ineq. (\ref{eq3.177}) can be expressed as
$$ks^2 \geq (s+k-1)^2,$$
which equivalent to
\begin{equation}\label{eq3.188}
s^2-2s-(k-1) \geq 0.
\end{equation}
In fact, Ineq. (\ref{eq3.188}) follows directly from the condition $s \geq \sqrt{k}+1$.  Hence, Ineq. (\ref{eq3.17}) holds. The Claim \ref{claim4} is proved.\\
On the other hand, Lemma \ref{lem2.2} and the assumption that $G$ is equiarboreal yield
\begin{equation}\label{eq3.21}
\Omega(G) =\frac{|V(G)|-1}{|E(G)|}=\frac{2(|V(G)|-1)}{k|V(G)|}<\frac{2}{k}.
\end{equation}
However, combining Claim \ref{claim4} with Ineq. (\ref{eq3.66}), we have
\begin{equation*}
 \Omega(G) \geq \frac{2}{k}.
\end{equation*}
which contradicts Ineq. (\ref{eq3.21}).  Therefore, we conclude that for $x,y \geq 1$ satisfying $ x+y \leq k-\sqrt{k}-2$, $G[C]$ is strongly $S_{x,y}$-free. The proof is complete.
\end{proof}

Let $C$ be an edge cut of $G$. For an integer $x \geq 3$, $G[C]$ is called \emph{strongly $S_x$-free}, if there does not exist vertex $u \in V(G[C])$ satisfying the following two conditions simultaneously: (i) $S_x$ is the subgraph of $G[C]$ induced by $N_{G[C]}[u]$; (ii) at least one leaf vertex of $S_x$ has degree 1 in $G[C]$. For example, the graph  as shown in Fig. \ref{Fig.2}(b) is not strongly $S_4$-free.

\begin{thm}\label{cor3.8}
Let $G$ be a connected $k$-regular equiarboreal graph.   Let $C$ be an edge cut of $G$.              Then \\
(i) If $k \geq 6$ and $|C| \leq k-1$, then for integer $x \geq 3$, $G[C]$ is strongly $S_{x}$-free.\\
(ii) If $k \geq 8$ and $C$ is a non-trivial edge cut with $|C| \leq k$, then for integer $x \geq 3$, $G[C]$ is strongly $S_{x}$-free.
\end{thm}
\begin{proof} We first prove the following claim.
\begin{cla}\label{claim1}
If $k \geq 6$ and $C$ is a non-trivial edge cut with $|C| \leq k$, then for $3 \leq x \leq k+\frac{4}{k}-3$, $G[C]$ is strongly $S_{x}$-free.
\end{cla}
\emph{Proof of Claim \ref{claim1}.} Let $C = \{uv\in E(G)\mid u\in A, v\in B \}$ be a non-trivial edge cut with $|C| \leq k$. Suppose, to the contrary, that there exists an integer $x_1$ with $3 \leq x_1 \leq k+\frac{4}{k}-3$ such that $G[C]$ is not strongly $S_{x_1}$-free. Without loss of generality, let $u_1$ be a vertex in $A$ such that $N_{G[C]}[u_1]$ induces a star $S_{x_1}$, and let $v_1$ be a vertex in $N_{G[C]}(u_1)$ with $d_{G[C]}(v_1) = 1$. For example, the structure of $G[C]$ is shown in Fig. \ref{Fig.4}(a).  By Theorem \ref{tm3.66}, we have
\begin{equation}\label{eq3.18}
\Omega(G) \geq \mathcal{F}(d_{G[C]}(u_1),d_{G[C]}(v_1)) = \mathcal{F}(x_1-1,1)=\frac{3k+x_1-5}{k^2+(x_1-1)k-2}.
\end{equation}
\begin{figure*}[ht]
  \setlength{\abovecaptionskip}{0cm} %
  \centering
 $$\includegraphics[width=1.4in]{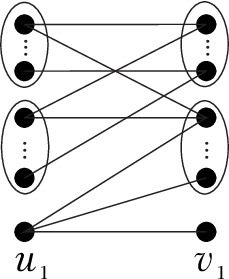}$$\\
 \caption{The structure example of $G[C]$ in the proof of Theorem \ref{cor3.8}.}    \label{Fig.4}
\end{figure*}
For convenience, define a function $\mathcal{G}(x)$ for $x \geq 3$ by
$$\mathcal{G}(x) = \dfrac{3k+x-5}{k^2+(x-1)k-2}.$$
By computing, we derive
$$\mathcal{G}'(x)=-\dfrac{2(k-1)^{2}}{(k^{2}+(x-1)k-2)^{2}}.$$
Consequently, $\mathcal{G}'(x) < 0$ for $x \geq 3$, which means $\mathcal{G}(x)$ is strictly decreasing for $x \geq 3$.\\
Setting $\mathcal{G}(x) = \frac{2}{k}$ leads to
\begin{equation} \label{eq3.9}
x=k+\frac{4}{k}-3.
\end{equation}
Combining Ineq. (\ref{eq3.18}) with Eq. (\ref{eq3.9}),  we conclude that
\begin{equation}\label{eq3.10}
\Omega(G) \geq \dfrac{2}{k}.
\end{equation}
Note that $G$ is an equiarboreal graph, by Lemma \ref{lem2.2}, we have
\begin{equation*}
\Omega(G) =\frac{2(|V(G)|-1)}{k|V(G)|}<\frac{2}{k}.
\end{equation*}
a contradiction to Ineq. (\ref{eq3.10}). Thus it follows, as desired, that for $3 \leq x \leq k+\frac{4}{k}-3$, $G[C]$ is strongly $S_{x}$-free. The Claim \ref{claim1} is holds.

Now we prove the statement (i). Let $C = \{uv\in E(G)\mid u\in A, v\in B \}$ be an edge cut with $|C| \leq k-1$. For the case where $x \geq k+1$, it is clear that $G[C]$ is strongly $S_{x}$-free. For the case where $3 \leq x \leq k-3$, $G[C]$ is strongly $S_{x}$-free by Claim \ref{claim1}. For the case where $ k-2 \leq x \leq k$, by contradiction, there exists an integer $x_1$ with $k-2 \leq x_1 \leq k$ such that $G[C]$ is not strongly $S_{x_1}$-free. Then $G[C]$ contains a star $S_{x_1}$ induced by vertex set $N_{G[C]}[u]$, where, without loss of generality, let $u \in A$. Let $m$ be the number of edges not belonging  to $S_{x_1}$ in $G[C]$. We have $m \leq 2$. If $m=0$, it follows from Lemma \ref{lem3.1} that it is invalid.  If $m=1$, let $u_1v_1$ be the edge not belonging to $S_{x_1}$ in $C$.  According to Theorem \ref{tm3.5}, we know that $G[C]$ is $K_2$ component-free. So vertex $v_1 \in N_{G[C]}(u)$, this implies that $G[C]$ contains a star $S_3$ induced by $N_{G[C]}[v_1]$, and $d_{G[C]}(u_1)=1$. Therefore, $G[C]$ is not strongly $S_{3}$-free, a contradiction to Claim \ref{claim1}. Hence, we only need to consider $m=2$, i.e., $x_1=k-2$.

\textbf{Case 1.} $k =6$. For this case, we have $|C|=|E(S_4)|+2=5$, which contradicts Lemma \ref{lem3.3}.

\textbf{Case 2.} $k \geq 7$. Let $e_1=u_1v_1$ and $e_2=u_2v_2$ be the edges not belonging  to $S_{k-2}$ in $G[C]$. We first consider $u_1\neq u_2$. If $|N_{G[C]}(u) \cap \{v_1,v_2\}|=0$, we have $v_1=v_2$ from Theorem \ref{tm3.5}. Then $G[C]$ contains a star $S_3$ induced by $N_{G[C]}[v_1]$, and $d_{G[C]}(u_1)=1$. This means that $G[C]$ is not strongly $S_{3}$-free, which contradicts Claim \ref{claim1}. If $|N_{G[C]}(u) \cap \{v_1,v_2\}| > 0$, without loss of generality, suppose that $v_1 \in N_{G[C]}(u)$.  Similarly, we can see that $C$ contains a star $S_3$ or $S_4$ induced by $N_{G[C]}[v_1]$, and $d_{G[C]}(u_1)=1$. Hence, $G[C]$ is not strongly $S_{3}$ or $S_{4}$-free, which contradicts Claim \ref{claim1}. We continue to consider $u_1=u_2$. If $|N_{G[C]}(u) \cap \{v_1,v_2\}| < 2$, then $G[C]$ contains a star $S_3$ induced by $N_{G[C]}[u_1]$, and at least one of the vertices $v_1$ and $v_2$ has a degree of 1 in $G[C]$. Hence, $G[C]$ is not strongly $S_{3}$-free, which contradicts Claim \ref{claim1}. If $|N_{G[C]}(u) \cap \{v_1,v_2\}| = 2$, we know that $u_1v_1 \in C$ and $d_{G[C]}(u_1)=d_{G[C]}(v_1)=2$, $G[C]$ is not strongly $S_{1,1}$-free, this contradicts Theorem \ref{tm3.13}. Therefore, the statement (i) holds.

Then we prove the statement (ii). Let $C = \{uv\in E(G)\mid u\in A, v\in B \}$ be a non-trivial edge cut with $|C| \leq k$. Similarly, for the case where $x \geq k+1$, obviously $G[C]$ is strongly $S_{x}$-free. For the case where $ 3 \leq x \leq k-3$,  by Claim \ref{claim1}, $G[C]$ is strongly $S_{x}$-free. For the case where $ k-2 \leq x \leq k$, assume by contradiction that $G[C]$ is not strongly $S_{x_1}$-free for an integer $x_1$ with $k-2 \leq x_1 \leq k$. Then $G[C]$ contains a star $S_{x_1}$ induced by the vertex set $N_{G[C]}[u]$, where we may assume $u \in A$. Let $m$ be the number of edges not belonging to $S_{x_1}$ in $G[C]$, we have $ m \leq 3$. The proof for the case of $ m \leq 2$ is similar to that of statement (i). When $m=3$ (i.e., $x_1=k-2$),  let $e_i=u_iv_i$ ($1 \leq i \leq 3$) denote an edge that does not belong to $S_{k-2}$ in $G[C]$.  We first consider $|\{u_1,u_2,u_3\}| > 1$. Without loss of generality, let $u_1 \neq u_2$ and $u_1 \neq u_3$. Then $d_{G[C]}(u_1)=1.$  By Theorem \ref{tm3.5}, we have $d_{G[C]}(v_1) \geq 2$. It is not hard to verify that $G[C]$ contains a star $S_t$ induced by $N_{G[C]}[v_1]$, where $3 \leq t \leq 5$, and since $d_{G[C]}(u_1)=1,$  $G[C]$ is not strongly $S_{t}$-free, where $3 \leq t \leq 5$. This  contradicts  Claim \ref{claim1}. We then consider $|\{u_1,u_2,u_3\}| = 1$.  If $|N_{G[C]}(u) \cap \{v_1,v_2,v_3\}| < 3$, then $G[C]$ contains a star $S_4$ induced by $N_{G[C]}[u_1]$, and at least one of the vertices $v_1$,  $v_2$ and $v_3$ has a degree of 1 in $G[C]$. Hence, $G[C]$ is not strongly $S_{4}$-free, which contradicts Claim \ref{claim1}.  If $|N_{G[C]}(u) \cap \{v_1,v_2, v_3\}| = 3$, we have $u_1v_1 \in C$, $d_{G[C]}(u_1)=3$ and $d_{G[C]}(v_1)=2$, $G[C]$ is not strongly $S_{2,1}$-free, which contradicts Theorem \ref{tm3.13}. Thus, we proof  the statement (ii).
\end{proof}

\begin{thm}\label{cor3.9}
Let $G$ be a connected $k$-regular equiarboreal graph with $k \geq 6$. Then for each edge cut $C$ of $G$ with $|C| \leq k-1$,  $G[C]$ does not contain  vertex of degree 1.
\end{thm}
\begin{proof}
Let $C$ be an edge cut of $G$ with $|C| \leq k-1$.  By contradiction, suppose that $G[C]$ contains a vertex of degree 1, denoted as $u$, and let $v$ be the unique neighbor of $u$. If $d_{G[C]}(v)=1$, then the edge $uv$ forms a $K_2$ component in $G[C]$, which contradicts Theorem \ref{tm3.5}. If $d_{G[C]}(v) \geq 2$, then $G[C]$ contains a star $S_x$ induced by $N_{G[C]}[v]$, where $x \geq 3$, and $d_{G[C]}(u)=1$. Therefore, $G[C]$ is not strongly $S_{x}$-free, which contradicts the statement (i) of Theorem \ref{cor3.8}. The proof is complete.
\end{proof}
Similarly, the following result can be obtained.
\begin{thm}\label{cor3.10}
Let $G$ be a connected $k$-regular equiarboreal graph with $k \geq 8$. Then for each non-trivial edge cut $C$ of $G$ with $|C| \leq k$,  $G[C]$ does not contain  vertex of degree 1.
\end{thm}
\subsection{$6$-regular equiarboreal  graphs and $7$-regular equiarboreal graphs}
\begin{thm}\label{tm3.10}
Let $G$ be a connected $6$-regular equiarboreal graph. Then $\lambda(G)=6$.
\end{thm}
\begin{proof}
Let $G$ be a connected 6-regular equiarboreal graph. By Theorem \ref{tm1.3}, we have $\lambda(G)  > 3.$ Meanwhile, $\lambda(G)\neq 5$ follows directly from Lemma \ref{lem3.3}. The only task is to consider the case where $\lambda(G)=4$. By Theorem \ref{cor3.9}, we know that for every non-trivial edge cut $C$ with $|C|=4$, $G[C]$ does not contain  vertex of degree 1. Hence, we only consider the edge cut case as shown in Fig. \ref{Fig.2}(d). For this case, similar to the proof of Case 3 in Theorem \ref{tm3.6}, we can show that $\Omega(G) \geq \frac{11}{28}.$ Whereas, according to Lemma \ref{lem2.2}, we have $\Omega(G) < \frac{1}{3}$, which is a contradiction. Consequently, we have $\lambda(G)=6$. The proof is complete.
\end{proof}

\begin{thm}\label{tm3.14}
Let $G$ be a connected $7$-regular equiarboreal graph. Then $\lambda(G)=7$.
\end{thm}
\begin{proof}
By contradiction, suppose that there exists an edge cut $C$ of $G$ with $|C| \leq 6$. Choose an edge $(u_1,v_1) \in C$. By Theorem \ref{cor3.9}, we have $d_{G[C]}(u_1) \geq 2$ and $d_{G[C]}(v_1) \geq 2$. Then $G[C]$ contains a double star $S_{x_1,y_1}$ centered at $u_1$ and $v_1$ as its subgraph (see Fig. \ref{Fig.8}), where $x_1= d_{G[C]}(u_1)-1$ and $y_1= d_{G[C]}(v_1)-1$.
\begin{figure*}[ht]
  \setlength{\abovecaptionskip}{0cm} %
  \centering
 $$\includegraphics[width=3.2in]{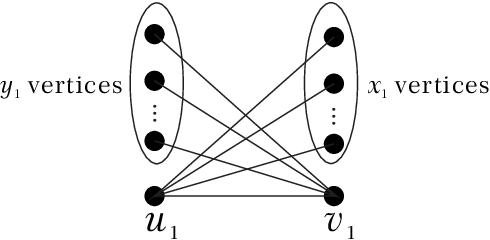}$$\\
 \caption{The double star $S_{x_1,y_1}$ in $G[C]$.}    \label{Fig.8}
\end{figure*}

\noindent On the other hand, according to Theorem \ref{tm3.13},  for all integers $x,y \geq 1$ satisfying $ x+y \leq 7- \sqrt{7}-2$, $G[C]$ is strongly $S_{x,y}$-free. Hence, we have
\begin{align}d_{G[C]}(u_1)+d_{G[C]}(v_1) &=x_1+1+y_1+1 \notag \\
& \geq \lfloor 7-\sqrt{7}-2 \rfloor +1+2 \notag \\
& =5. \notag
\end{align}
Consequently, we have $x_1+y_1 \geq 3$, then we show that $x_1+y_1=3$.  If $x_1+y_1 \geq 4$, we have
$$ \mbox{max} \{|N_{G[C]}(u_1)\setminus \{v_1\}|, |N_{G[C]}(v_1)\setminus \{u_1\}| \} \geq 2.$$
However, there is at most one edge in $G[C]$ that does not belong to $S_{x_1,y_1}$. This implies at least one vertex in $N_{G[C]}(u_1)$ or $N_{G[C]}(v_1)$ has  degree 1 in $G[C]$, a contradiction to Theorem \ref{cor3.9}. Without loss of generality, let $x_1=2$ and $y_1=1$. Observe that $G[C]$ does not contain vertex of degree 1, and there are at most two edges in $G[C]$ that do not belong to $S_{2,1}$.  It follows that $|C|=6$ and $G[C]$ is a complete bipartite graph $K_{2,3}$ with partite sets $\{u_1,u_2\}$ and $\{v_1,v_2,v_3\}$. \\
\begin{figure*}[ht]
  \setlength{\abovecaptionskip}{0cm} %
  \centering
 $$\includegraphics[width=4.8in]{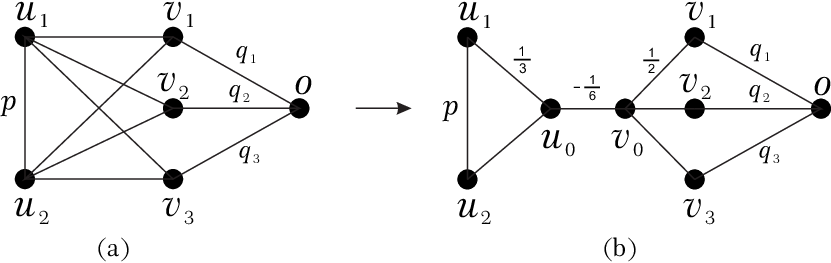}$$\\
 \caption{(a) The network $N$ equivalent to $G$, (b) The network $N_1$ equivalent to $N$.}    \label{Fig.7}
\end{figure*}
Recall that $G_1 = G[A]$ and $G_2 = G[B]$. Now we construct a network $N$ as shown in Fig. \ref{Fig.7}(a), with the weights satisfying:\\
(i) $p= R_{N}(u_1,u_2) =\Omega_{G_1}(u_1,u_2)$;\\
(ii) $R_{N}(u_i,v_j)=1$ for $1 \leq i  \leq 2$ and $1 \leq j \leq 3$;\\
(iii) $q_i = R_{N}(v_i, o)$ for $1 \le i \le 3$, and $q_i + q_j = \Omega_{G_2}(v_i, v_j)$ for $1 \le i \ne j \le 3$.\\
Through simple calculations, we have
\begin{equation}\label{eq3.233}
\left\{
\begin{aligned}
 q_1&=\dfrac{\Omega_{G_2}(v_1,v_2)+\Omega_{G_2}(v_1,v_3)-\Omega_{G_2}(v_2,v_3)}{2},  \\
 q_2&=\dfrac{\Omega_{G_2}(v_1,v_2)+\Omega_{G_2}(v_2,v_3)-\Omega_{G_2}(v_1,v_3)}{2},\\
 q_3&=\dfrac{\Omega_{G_2}(v_1,v_3)+\Omega_{G_2}(v_2,v_3)-\Omega_{G_2}(v_1,v_2)}{2}.
\end{aligned}
\right.
\end{equation}
From the triangle inequality property of resistance distances, then each $q_i$ ($1\leq i \leq 3$) has a unique non-negative solution. By definition, we know that $G_2$ and the star network induced by the vertex set $\{o,v_1,v_2,v_3\}$ in $N$ are $\{v_1,v_2,v_3\}$-equivalent networks.  Using the principle of substitution, we can see that $G$ and $N$ are $\{u_i,v_j | 1\leq i \leq 2, 1\leq j \leq 3\}$-equivalent networks. Then, applying the complete bipartite graph-double star transformation, we obtain the equivalent network $N_1$ as shown in Fig. \ref{Fig.7}(b), with the new edges satisfying:
$$R_{N_1}(u_0,u_i)=\frac{1}{3}, \quad R_{N_1}(v_0,v_j)=\frac{1}{2}, \quad R_{N_1}(u_0,v_0)=-\frac{1}{6}.$$
where $1\leq i \leq 2$,  $1\leq j\leq 3.$\\
We prove the following claim.
\begin{cla}\label{claim2} $q_1=q_2=q_3$.
\end{cla}
\textit{Proof of Claim \ref{claim2}.} For $1 \leq i \leq 3$, let $t_i=\frac{1}{2}+q_i$. It is sufficient to prove that $t_1=t_2=t_3$. Here we only prove that $t_1=t_2$, the case for $t_1=t_3$ can be proved in the same way.  Since $G$ is an equiarboreal graph, we have $\Omega_{N_1}(u_1,v_1)=\Omega_{N_1}(u_2,v_2)$. By the series and parallel connections, this is equivalent to
$$q_1+\dfrac{t_2t_3}{t_2+t_3}=q_2+\dfrac{t_1t_3}{t_1+t_3},$$
which yields that
\begin{align}
q_1&+\frac{t_2t_3}{t_2+t_3}-\left(q_2+\frac{t_1t_3}{t_1+t_3}\right)  \notag\\
&=  t_1+\frac{t_2t_3}{t_2+t_3}-\left(t_2+\frac{t_1t_3}{t_1+t_3}\right)     \notag\\
&= (t_1-t_2)+t_3\left(\frac{t_2}{t_2+t_3}-\frac{t_1}{t_1+t_3}\right)\notag\\
 &= (t_1-t_2)+t_3\left(\frac{t_3(t_2-t_1)}{(t_2+t_3)(t_1+t_3)}\right)        \notag\\
&=(t_1-t_2)\left(1-\frac{t_3^2}{(t_2+t_3)(t_1+t_3)}\right)\notag\\
&=(t_1-t_2)\left(\frac{t_1t_2+t_2t_3+t_3t_1}{(t_2+t_3)(t_1+t_3)}\right) \notag \\
&=0.\notag
\end{align}
Accordingly, we get $t_1=t_2$. The Claim \ref{claim2} is proved.\\
Since the degree of $v_i$ ($1 \leq i \leq 3$) in $G_2$ is $5$,  by Lemma \ref{lem2.6}, we get
$$ \Omega_{G_2}(v_i,v_j) \geq \dfrac{1}{d_{G_2}(v_i)+1} + \dfrac{1}{d_{G_2}(v_j)+1} =\frac{1}{3},   $$
where $1 \leq i \neq j \leq 3$.  Together with Eq. (\ref{eq3.233}) and Claim \ref{claim2}, we have
$$q_1=q_2=q_3=\frac{\Omega_{G_2}(v_1,v_2)+\Omega_{G_2}(v_1,v_3)+\Omega_{G_2}(v_2,v_3)}{6} \geq \frac{1}{6}.$$
Similarly, we have $\Omega_{G_1}(u_1,u_2) \geq \frac{2}{5}$.
We now construct a corresponding  network $N_2$ obtained from $N_1$ by changing the weight $p$ to $\frac{2}{5}$, and changing the weights $q_i$ ($1 \leq i \leq 3$) to $\frac{1}{6}$, respectively, and the weights of other edges remain unchanged.  Since both $u_0$ and $v_0$ are cut-vertices of $N_2$, it follows from the cut-vertex property of resistance distances that
\begin{align}\label{eq3.22}
\Omega_{N_2}(u_1,v_1) &=\Omega_{N_2}(u_1,u_0)+\Omega_{N_2}(u_0,v_0)+\Omega_{N_2}(v_0,v_1) \notag \\
&=\frac{11}{48}-\frac{1}{6}+\frac{1}{4}\notag \\
&=\frac{5}{16}.
\end{align}
Based on the principle of substitution and Rayleigh's monotonicity law, we have
\begin{equation}\label{eq3.24}
\Omega(G)=\Omega_G(u_1,v_1)=\Omega_{N_1}(u_1,v_1) \geq \Omega_{N_2}(u_1,v_1) = \frac{5}{16}.
\end{equation}
Furthermore, by Lemma \ref{lem2.2}, we have
\begin{equation*}
\Omega(G) =\frac{|V(G)|-1}{|E(G)|}=\frac{2(|V(G)|-1)}{7|V(G)|}<\frac{2}{7}.
\end{equation*}
which contradicts Ineq. (\ref{eq3.24}). We complete the proof.
\end{proof}

\subsection{$k$-regular equiarboreal graphs with $k \geq 8$}
\begin{thm}\label{tm3.15}
Let $G$ be a connected $k$-regular equiarboreal graph with $k \geq 8$. Then $\lambda(G)=k$. In particular, if $k \geq 11$, then each edge cut $C$ with $|C|=k$ is trivial.
\end{thm}
\begin{proof}
 We have the following claim.
\begin{cla}\label{claim3}
Let $C$ be a non-trivial edge cut of $G$ with $|C| \leq k$. Then $|C| \geq 2 \lfloor k-\sqrt{k}\rfloor -2.$
\end{cla}
\emph{Proof of Claim \ref{claim3}.}
Choose an edge $(u_1,v_1) \in C$,  according to Theorem \ref{cor3.10},  we know that $d_{G[C]}(u_1) \geq 2$ and $d_{G[C]}(v_1) \geq 2$. Hence, $G[C]$ contains a double star $S_{x_1,y_1}$ centered at $u_1$ and $v_1$ as its subgraph (see Fig. \ref{Fig.8}), where $x_1 = d_{G[C]}(u_1)-1$ and $y_1 = d_{G[C]}(v_1)-1$.

Let $w$ be a vertex with the smallest degree in $N_{G[C]}(v_1)$, i.e.,
$$  w \in \left\{ u \in N_{G[C]}(v_1) \mid d_{G[C]}(u) = \min_{v \in N_{G[C]}(v_1)} d_{G[C]}(v) \right\}.   $$
Again using  Theorem  \ref{cor3.10}, we have $d_{G[C]}(w) \geq 2$.  Then by Theorem \ref{tm3.13}, for integers $x,y \geq 1$ satisfying $ x+y \leq k- \sqrt{k}-2$ , $G[C]$ is strongly $S_{x,y}$-free. Hence, we know that
\begin{equation}\label{eq3.23}
d_{G[C]}(w)+d_{G[C]}(v_1) \geq \lfloor k-\sqrt{k}-2 \rfloor +1+2=  \lfloor k-\sqrt{k} \rfloor +1.
\end{equation}
Since $d_{G[C]}(w) \geq 2$ and $d_{G[C]}(v_1) \geq 2$, together with Eq. (\ref{eq3.23}), we estimate
\begin{align}\label{Ineq3.25}
|C| =|E(G[C])| &\geq \sum_{u \in N_{G[C]}(v_1)}d_{G[C]}(u) \notag \\
&\geq \sum_{u \in N_{G[C]}(v_1)}d_{G[C]}(w) \notag \\
&=d_{G[C]}(w)\cdot d_{G[C]}(v_1) \notag \\
& \geq 2\cdot (\lfloor k-\sqrt{k} \rfloor -1) \notag \\
& = 2 \lfloor k-\sqrt{k}\rfloor -2.\notag
\end{align}
The Claim \ref{claim3} is holds.\\
Then we prove that $\lambda(G)=k$ for $k \geq 8$.  By contradiction, suppose there exists an edge cut $C$ with $|C| \leq k-1$. For $k \geq 8$, according to Claim \ref{claim3}, it is not hard to verify that
$$ |C| \geq 2 \lfloor k-\sqrt{k}\rfloor -2 \geq k,$$
achieving a contradiction with the assumption that $|C| \leq k-1$. Hence, $G$ does not contain edge cut $C$ with $|C| \leq k-1$, which means that $\lambda(G)=k$.

Finally, we show that if $k \geq 11$, then each edge cut $C$ with $|C|=k$ is a trivial edge cut. Again by contradiction, assume that there exists a non-trivial edge cut $C$ with $|C|=k$. Similarly, by Claim \ref{claim3}, we have
$$ |C| \geq 2 \lfloor k-\sqrt{k}\rfloor -2 \geq k+1,$$
achieving a contradiction with the assumption that $|C|=k$. Therefore, such edge cut $C$ does not exist. We conclude that $C$ is a trivial edge cut if $|C|=k$. The proof is complete.
\end{proof}

Combining Theorems \ref{tm3.1}, \ref{tm3.55}, \ref{tm3.6}, \ref{tm3.10}-\ref{tm3.15}, we give the main result.
\begin{thm}\label{tm3.17}
Let $G$ be a connected $k$-regular equiarboreal graph with $k \geq 1$. Then $\lambda(G)=k$. In particular, if $k \geq 11$, then each edge cut $C$ with $|C|=k$ is trivial.
\end{thm}
As a by-product of  Theorems \ref{thm1.2} and \ref{tm3.17}, we confirm the Conjecture \ref{coj1}.
\begin{thm}
Let $G$ be a connected graph which is a colour class in an association scheme. Then its edge-connectivity $\lambda(G)$ equals  its degree.
\end{thm}

It is known that a $k$-regular graph with edge-connectivity at least $k-1$ on an even number of vertices contains a perfect matching \cite{cbe}. By Theorem \ref{tm3.17}, we obtain the following result on the existence of perfect matchings in regular equiarboreal graphs, which generalizes the corresponding results for distance-regular graphs \cite{aeb1}.

\begin{cor}
Let $G$ be a connected $k$-regular equiarboreal graph on an even number of vertices. Then $G$ contains a perfect matching.
\end{cor}

\section{Conclusions}\label{section4}
In this paper, we study the edge-connectivity of regular equiarboreal graphs and further prove that the edge-connectivity of a connected regular equiarboreal graph equals its degree, which confirms the Godsil's conjecture. It is worth mentioning that Brouwer \cite{aeb0} made the stronger conjecture that for any connected graph which is a colour class in an association scheme, its vertex-connectivity equals its degree. Brouwer's conjecture has been proved by Brouwer and Koolen \cite{aeb2} for distance-regular graphs, but remains an open problem in the other cases.

\section{Declaration of competing interest}
The authors declare that they have no known competing financial interests or personal relationships that could have appeared to influence the work reported in this paper.
\section{Data availability}
No data was used for the research described in the article.
\section{Acknowledgments}
The second author is supported by the National Natural Science Foundation of China (through grant No. 12171414) and Taishan Scholars Special Project of Shandong Province (through grant No. tsqn202211113). The third author is support by the National Natural Science Foundation of China (through grant No. 12071194).

\begin{appendices}
\section{The elaborated simplification of Eq. (\ref{eq3.88}) in the proof of Theorem \ref{tm3.66}.}\label{appb}
\begin{proof}
According to Eq. (\ref{eq3.15}), we have
\begin{align}
a V_{u_{2}} + (x-1) V_{v_{2}}&=a \cdot \frac{2ab + c(x-1)}{4ab - c^{2}}+(x-1) \cdot \frac{a(2x + c - 2)}{4ab - c^{2}}\notag\\
& = \frac{ a(2ab + c(x-1)) + a(x-1)(2x + c - 2) }{4ab - c^{2}} \notag\\
&= \frac{ 2a \left[ ab + c(x-1) + (x-1)^2 \right] }{4ab - c^{2}}. \notag
\end{align}
Since
$$I = k - \left( a V_{u_{2}} + (x-1) V_{v_{2}} \right),$$
which yields that
\begin{equation}\label{eqA.1}
I = \frac{ k(4ab - c^{2}) - 2a \left[ ab + c(x-1) + (x-1)^2 \right] }{4ab - c^{2}}.
\end{equation}
By expanding the numerator of  Eq. (\ref{eqA.1}), we have
$$I = \frac{ 4k a b - k c^{2} - 2a^2 b - 2a c (x-1) - 2a (x-1)^2 }{4ab - c^{2}}.$$
In order to obtain Eq. (\ref{eq3.99}), it suffices to prove that
$$4k a b - k c^{2} - 2a^2 b - 2a c (x-1) - 2a (x-1)^2 = 2k a b + 2 a b - k c^{2}.$$
In fact, we can show
\begin{align}
4k a b - k c^{2} - 2a^2 b& - 2a c (x-1) - 2a (x-1)^2-(2k a b + 2 a b - k c^{2})\notag\\
&=2k a b - 2a^2 b - 2 a b - 2a c (x-1) - 2a (x-1)^2\notag\\
&=2a[k b - a b - b - c (x-1) - (x-1)^2]\notag\\
&=2a[b(k - a - 1) - c(x-1) - (x-1)^2] \notag \\
&=2a[(x-1)(b - c - (x-1))] \quad (\mbox {since} \ k-a-1=x-1) \notag \\
&=2a(x-1)(b - c - (x-1)) \notag \\
&=0.  \quad (\mbox {since} \ b - c = (k - y) - (k - x - y + 1) = x - 1)\notag
\end{align}
Hence, we obtain desired result. The proof is complete.
\end{proof}
\end{appendices}

\begin{appendices}
\section{Justification of Eq. (\ref{eq3.16}) in the proof of Theorem \ref{tm3.13}.}\label{appa}
\begin{proof}Recall that
$$\mathcal{F}(x+1,y+1)=\dfrac{4(k-x-1)(k-y-1)-(k-x-y-1)^2}{2(k-x-1)(k-y-1)(k+1)-k(k-x-y-1)^{2}}.$$
To obtain the desired result, we prove that the denominator of $\mathcal{F}(x+1,y+1)$ is strictly greater than 0, under the conditions $k \geq 7$, $x,y \geq 1$ and $ x+y \leq k-1$ (which includes the case $x+y \leq k-\sqrt{k}-2$).

According to the proof of Theorem \ref{tm3.66}, we know that $\mathcal{F}(x+1, y+1)$ is a function of the resistance distance between vertices $u_1$ and $v_1$ in network $N$. With the aid of the non-negative property of resistance distance, it is sufficient to prove that the numerator of $\mathcal{F}(x+1,y+1)$ is positive. That is,
\begin{equation*}\label{eqA}
4(k-x-1)(k-y-1)- (k-x-y-1)^2 > 0.
\end{equation*}
Let $a=k-x-1$, $b=k-y-1$ and
$$\mathcal{H}(a,b)=4ab-(a+b-(k-1))^{2}.$$
In addition, we have $a,b \in [1,k-2]$ and $a+b \in [k-1,2k-4]$. To prove that $\mathcal{H}(a,b) >0,$ it is equivalent to show that
$$ 4ab >  (a+b-(k-1))^{2}.   $$
Since $a+b \geq k-1$, taking square roots (positive roots), we have
$$2\sqrt{ab} > a+b-(k-1),$$
which yields that
\begin{equation}\label{eqB}
(\sqrt{a}-\sqrt{b})^2 < k-1.
\end{equation}
In the following, we show that  Ineq. (\ref{eqB}) holds. Since \( a, b \in [1, k-2] \), the maximum value of the difference \( |\sqrt{a} - \sqrt{b}| \) is attained at \( a = 1 \), \( b = k-2 \) or vice versa, we have
\[
\max |\sqrt{a} - \sqrt{b}| = \sqrt{k-2} - 1
\]
By computing, we have
\[
\max (\sqrt{a} - \sqrt{b})^2 = (\sqrt{k-2} - 1)^2 =k - 1 - 2\sqrt{k-2}.
\]
Therefore, we have
\[
(\sqrt{a} - \sqrt{b})^2 \leq k - 1 - 2\sqrt{k-2}< k-1,
\]
due to $k \geq 7$. The proof is complete.
\end{proof}
\end{appendices}
\end{CJK}

\end{document}